\documentclass[leqno]{article}
\usepackage{amsfonts}
\usepackage{amssymb}
\usepackage{amsmath}
\usepackage{amsthm}
\usepackage{dsfont}
\usepackage{mathtools}
\usepackage[utf8]{inputenc}
\usepackage{graphicx}
\usepackage{notation}
\usepackage{float} 
\usepackage{hyperref}

\begin{document}


\theoremstyle{definition}
\newtheorem{corollary}{Corollary}
\newtheorem{definition}[corollary]{Definition}
\newtheorem{lemma}[corollary]{Lemma}
\newtheorem{notation}[corollary]{Notation}
\newtheorem{proposition}[corollary]{Proposition}
\newtheorem{theorem}[corollary]{Theorem}

\oddsidemargin 16.5mm
\evensidemargin 16.5mm

\vspace{5cc}
\begin{center}

{\large\bf  A CONNECTION BETWEEN THE $A_\alpha$-SPECTRUM AND LOVÁSZ THETA NUMBER
\rule{0mm}{6mm}\renewcommand{\thefootnote}{}
\footnotetext{\scriptsize ${}^{\ast}$Corresponding author. Gabriel Coutinho}
\footnotetext{\scriptsize 2020 Mathematics Subject Classification. 05C50, 90C22.

\rule{2.4mm}{0mm}Keywords and Phrases. Lovász theta function, $A_\alpha$-spectrum.}}

\vspace{1cc}
{\large\it Gabriel Coutinho${}^{\ast}$, Thiago Oliveira}

\vspace{1cc}
\parbox{24cc}{{\small

We show that the smallest $\alpha$ so that $\alpha D + (1-\alpha)A \succcurlyeq 0$ is at least $1/\vartheta(\ov{G})$, significantly improving upon a result due to Nikiforov and Rojo (2017). In fact, we display a stronger connection: if the nonzero entries of $A$ are allowed to vary and those of $D$ vary accordingly, then we show that this smallest $\alpha$ is in fact equal to $1/\vartheta(\ov{G})$. We also show other results obtained as an application of this optimization framework, including a connection to the well-known quadratic formulation for $\omega(G)$ due to Motzkin and Straus (1964). 

\vspace{0.5cm}

\begin{center}
\textsl{In memoriam of Prof. Vladimir Nikiforov.}
\end{center}

}}
\end{center}


\vspace{1.5cc}
\begin{section}
{Introduction}
\end{section}

Let $G$ be a simple graph, and define $A = A(G)$ to be its adjacent matrix, and $D = D(G)$ to be the diagonal matrix of its vertex degrees. For any $\alpha \in [0,1]$, let 
\[
	A_\alpha \coloneqq \alpha D + (1-\alpha)A.
\]
This parametrized family of matrices was introduced in \cite{Nikiforov17a} and has since received considerable attention (see for instance \cite{LiCM19a, ShuchaoW20a, NikiforovPRS17a, WangWT20a, XueLLS18a}).

In \cite{NikiforovR17a}, Nikiforov and Rojo defined $\alpha_0$ to be the smallest real number so that $A_\alpha$ is positive semidefinite, that is $A_\alpha \succcurlyeq 0$, for all $\alpha \geq \alpha_0$. They showed a closed expression for $\alpha_0$ when $G$ is regular depending only on the eigenvalues of $A$, characterized that $\alpha_0 = 1/2$ if and only if $G$ is bipartite, and showed that for all graphs $\alpha_0 \geq 1/\chi(G)$, where $\chi(G)$ stands for the chromatic number of $G$. Following their work, some articles \cite{BrondaniFO22a, BrondaniOFL19a} studied the value of $\alpha_0$ for some classes of graphs. 

In this short paper, we present a connection between $\alpha_0$ and $\vartheta(\ov{G})$, the Lovász theta number of $\ov{G}$, also known as the strict vector chromatic number of $G$. We recall the well known fact that $\vartheta(\ov{G}) \leq \chi(G)$ (see \cite{LovaszShanon}).

\vspace{1.5cc}
\begin{section}
{Semidefinite programming formulation} \label{sec:sdpfor}

\end{section}

In this section, we fix the graph $G$ with adjacency matrix $A$ and diagonal degree matrix $D$. Let $L = D-A$ be the Laplacian matrix of $G$. Recall that for two square matrices $M$ and $N$, the trace inner product is given by
\[
	\langle M,N \rangle = \tr M^T N.
\]
We use \(I\) and \(J\) to represent the identity and the all ones matrix of appropriate size, respectively.
Let $\Rds_+$ be the set of nonnegative reals.
The parameter $\alpha_0$ is the optimum of a semidefinite program with variable $\alpha \in \Rds_+$:
\begin{equation}
    \label{eq:alpha-def}
	\alpha_0 = \min \{ \alpha \in \Rds_+ : \alpha D + (1-\alpha)A \succcurlyeq 0.\}
\end{equation}
This was not explicitly stated in \cite{NikiforovR17a} but this is not an inconsequential fact: the optimum of semidefinite programs can be computed efficiently up to precision \cite{GroetschelLS81a, GroetschelLS93a}, thus, in principle, recovering $\alpha_0$ for a given graph is an easy task.
There is a plethora of efficient softwares available for semidefinite programs.
We highlight CSDP \cite{BorchersCSDP} and SeDuMi \cite{SturmSeDuMi}.

Analogously to the theory of linear programming, every semidefinite programming has a dual (see \cite[Section 4]{Helmberg02a}). The dual program to \eqref{eq:alpha-def} has variable $X \in \Sds_n$, a symmetric $n\times n$ matrix, and is given by
\begin{equation}
    \label{eq:dual}
	\sup \{\langle X,-A \rangle : X \succcurlyeq 0,\ \langle X,L\rangle \leq 
            1 \}.
\end{equation}
One can easily check that weak duality holds, that is, any objective value of \eqref{eq:dual} is upper bounded by the optimum value of \eqref{eq:alpha-def}:
\begin{align*}
	\langle X,-A \rangle & \leq \langle X, \alpha (D-A) \rangle \\
		& = \alpha \langle X, L\rangle \\
		& \leq \alpha.
\end{align*}
The first inequality follows from a result due to Schur (see \cite[Theorem 2.4]{Helmberg02a}), as both $X \succcurlyeq 0$ and $\alpha(D-A) + A \succcurlyeq 0$.

We remark here that if the graph \(G\) is disconnected, then the
matrices \(A, L\), and \(D\) have a block-diagonal decomposition and
the SDPs presented are equivalent to having an SDP for each
connected component.

For linear programs, it always follows that if both the primal and the dual programs are bounded, then their optima are the same. For semidefinite programs, this is achieved provided certain special condition holds, which we verify below.

\begin{proposition}
	The optimization problems in \eqref{eq:alpha-def} and \eqref{eq:dual} have
    the same optimal value.
\end{proposition}
\begin{proof}
    We will use \cite[Theorem 4.3]{Helmberg02a}.
    Let \(m\) be the number of edges of \(G\).
    The matrix
    \[X \coloneqq \frac{1}{2m + 1} \cdot I\] 
    is strictly feasible for \eqref{eq:dual}
    and the optimum value of \eqref{eq:dual} is upper bounded by \(\alpha_0\)
    by weak duality, so it is finite.
\end{proof}

The framework of semidefinite programming and the primal-dual formulations provide a systematic way in which bounds can be found, as we exemplify below.

\begin{proposition}
    \label{pro:degree-bound}
    Let \(G\) be a graph with at least one edge, minimum degree $\delta(G)$ and maximum degree $\Delta(G)$. Then
    \[
        \frac{-\lambda_{\min}(A)}{\Delta(G) - \lambda_{\min}(A)} \leq
        \alpha_0(G) \leq
        \frac{-\lambda_{\min}(A)}{\delta(G) - \lambda_{\min}(A)}.
    \]
    Consequently, if \(G\) is \(d\)-regular,
    \[
        \alpha_0(G) = \frac{-\lambda_{\min}(A)}{d - \lambda_{\min}(A)}.
    \]
\end{proposition}
\begin{proof}
    We start by proving the upper bound.
    We claim that
    \(\alpha \coloneqq \tfrac{-\lambda_{\min}(A)}
    {\delta(G) - \lambda_{\min}(A)}\) is feasible for \eqref{eq:alpha-def}.
    To show that \(\alpha D + (1 - \alpha)A \succcurlyeq 0\), we will
    use the fact that \(D \succcurlyeq \delta(G)I\).
    Thus, it is sufficient to prove \(\alpha \delta(G) I + (1 - \alpha) A
    \succcurlyeq 0\), or equivalently, \(\alpha \delta(G) \geq
    \lambda_{\max}((\alpha - 1)A)\).
    Note that
    \begin{align*}
        \alpha \delta(G) &=
        \frac{-\lambda_{\min}(A) \delta(G)}{\delta(G) - \lambda_{\min}(A)}\\
        &=
       \frac{-\delta(G)}{\delta(G) - \lambda_{\min}(A)} \cdot
        \lambda_{\min}(A) \\
        &=
        (\alpha - 1)\lambda_{\min}(A)
    \end{align*}
    and since \((\alpha - 1) \leq 0\), we have
    \((\alpha - 1)\lambda_{\min}(A) = \lambda_{\max}((\alpha - 1)A)\).

    We now prove the lower bound by exhibiting a feasible solution to the 
    dual problem \eqref{eq:dual}. Let \(v\) be a unit \(\lambda_{\min}(A)\)-eigenvector of \(A\).
    We can verify that
    \[
    \langle v^{}v^{T} , L \rangle =
    \langle v^{}v^{T} , D - A \rangle \leq
    \Delta(G) - \lambda_{\min}(A).
    \]
    Therefore, we have that
    \[
    Y \coloneqq \frac{v^{} v^{T}}
    {\Delta(G) - \lambda_{\min}(A)}
    \]
    is feasible for \eqref{eq:dual} and its objective value is
    \(\tfrac{-\lambda_{\min}(A)}{\Delta(G) - \lambda_{\min}(A)}\).
\end{proof}

\vspace{1.5cc}
\begin{section}
{Connection to the Lovász theta number}
\end{section}

In \cite{LovaszShanon}, Lovász introduced a parameter $\vartheta(G)$ which, for a fixed graph, is the optimum of a semidefinite program. We use $M\circ N$ to denote the entry-wise product of two matrices of the same size, and use $\ov{A}$ for the adjacency matrix of the complement graph $A(\ov{G})$.

The primal and dual formulations for $\vartheta$ are as follows
\begin{equation}
	\vartheta(G) = \max \{\langle X,J\rangle : X \circ A = 0,\ \langle X , I \rangle =1,\ X \succcurlyeq 0\}, \label{eq:varthetamax}
\end{equation}
and
\begin{equation}
	\vartheta(G) = \min \{ \lambda : \lambda I + Z - J \succcurlyeq 0,\ Z \circ (I + \ov{A}) = 0 \}.
\end{equation}
Lovász showed that $\vartheta(\ov{G}) \leq \chi(G)$, thus the following result is a strengthening of \cite[Theorem 10]{NikiforovR17a} and \cite[Corollary 11]{NikiforovR17a}.

\begin{theorem} \label{thm:easy}
    Let \(G\) be a graph with $n$ vertices and $m$ edges, $m > 0$. Then
    \[\lambda_{\min} (A_\alpha) \leq \frac{2m}{n} \cdot\frac{\vartheta(\ov{G}) \alpha - 1}{\vartheta(\ov{G}) - 1},\]
    and therefore
	\[\alpha_0 \geq \frac{1}{\vartheta(\ov{G})}.\]
\end{theorem}
\begin{proof}
	By definition of the smallest eigenvalue, once $\alpha$ is fixed, 
	\begin{equation}
	\lambda_{\min} (A_\alpha) = \max \{ \mu \in \Rds: A_\alpha - \mu I  \succcurlyeq 0\}. \label{eigenprimal}
	\end{equation}
	This too is a semidefinite program, whose dual is given by
	\begin{equation}
	\lambda_{\min} (A_\alpha) = \min \{ \langle X , A_\alpha \rangle : X \succcurlyeq 0, \langle X , I \rangle = 1\}. \label{eigendual}
	\end{equation}
    For more about the relation between eigenvalue optimization and semidefinite program, see  \cite[Section 6]{Helmberg02a}.
	Let $(\lambda,Z)$ be an optimum for the $\min$ formulation of $\vartheta(\ov{G})$, in particular $\lambda = \vartheta(\ov{G})$, $Z \circ (I + A) = 0$, and $(\vartheta(\ov{G})\cdot I + Z - J) \succcurlyeq 0$. Choose
	\[
		X = \frac{1}{n (\vartheta(\ov{G}) - 1)} \big(\vartheta(\ov{G})\cdot I + Z - J \big).
	\]
	Because $G$ has at least one edge, $\vartheta(\ov{G}) \geq 2$, so $X$ is well-defined, $X \succcurlyeq 0$, and it is immediate to check that $\langle X , I \rangle = 1$.
	
	For the objective value, recall that $\langle I,A\rangle = 0$ because $A$ has trace $0$, $\langle J , L \rangle = 0$ because the all ones vector is a $0$-eigenvector of $L$, and also $\langle Z , A_\alpha \rangle = 0$ because these matrices have disjoint support. Then
	\begin{align*}
		\langle X , A_\alpha \rangle & = \left\langle \frac{1}{n (\vartheta(\ov{G}) - 1)} \big(\vartheta(\ov{G})\cdot I + Z - J \big) , \alpha L + A \right\rangle \\
		& = \frac{1}{n (\vartheta(\ov{G}) - 1)} \big( \vartheta(\ov{G}) \langle I, \alpha L\rangle - \langle J,A\rangle \big) \\
		& = \frac{1}{n (\vartheta(\ov{G}) - 1)} \big( 2m \vartheta(\ov{G}) \alpha - 2m  \big), 
	\end{align*}
	as we wanted. For the second part, if $\alpha = \alpha_0$, then $\lambda_{\min} (A_\alpha) = 0$, and a simple reorganization implies the desired inequality.
\end{proof}

\vspace{1.5cc}
\begin{section}
{The parameter $\alpha_0$ for weighted graphs}
\end{section}

The previous section motivates the question of how $\alpha_0$ behaves if the matrices $D$ and $A$ have weighted entries. We can show that this optimized version of $\alpha_0$ will coincide precisely with $1 / \vartheta(\ov{G})$.

Define
\sdp{\widetilde{\alpha_0} \coloneqq \qquad\qquad \min }{\alpha \nonumber}{Y,Z \in \Sds_n,\ \alpha \in \Rds, \\ & \alpha Y + (1-\alpha)Z \succcurlyeq 0, \label{eq1}\\ & Y \circ I = Y, \label{eq:I}\\ & Z \circ A = Z, \label{eq:A}\\ & \langle J,Y \rangle = 1, \label{eq:norm}\\ & \langle J , Y-Z\rangle = 0.\label{eq:lapl}}
Equations \eqref{eq:I} and \eqref{eq:A} say that $Y$ is diagonal and $Z$ is nonzero only at entries corresponding to edges of $G$. Equation \eqref{eq:norm} is a normalization. Note that this is fine as \eqref{eq1} is scaling invariant, but adds the benefit of preventing the trivial situation in which $Y=Z=0$. Finally \eqref{eq:lapl} is analogous to the fact that $\langle J,D-A\rangle = 0$, and a weakening of the fact that the entries of $D$ are the row-sums of $A$. In a sense, one can see $\widetilde{\alpha_0}$ as the smallest possible $\alpha_0$ for a nonzero weighted subgraph of $G$.

The semidefinite programming framework allows for an immediate proof of the following theorem.

\begin{theorem} \label{thm:main}
    Let \(G\) be a graph with at least one edge. Then
	\[
		\widetilde{\alpha_0} = \frac{1}{\vartheta(\ov{G})}.
	\]
\end{theorem}
\begin{proof}

A simple renormalization of \eqref{eq:varthetamax}, taken for the complement graph, leads to the program

\begin{equation}
	\frac{1}{\vartheta(\ov{G})} = \min \{\langle X,I\rangle : X \circ \ov{A} = 0,\ \langle X , J \rangle =1,\ X \succcurlyeq 0\}. \label{eq:1var}
\end{equation}

	Now let $\alpha',Y',Z'$ be optima for the formulation of $\widetilde{\alpha_0}$, and make $W = \alpha' Y' + (1-\alpha') Z'$. Note that
	\[
		W \succcurlyeq 0, \quad  W \circ \ov{A} = 0 , \quad\text{ and }\quad \langle J,W \rangle = 1,
	\]
	thus $W$ is feasible for \eqref{eq:1var} and its objective value is $\widetilde{\alpha_0}$. 
	
	On the other hand, if $X$ is an optimum for \eqref{eq:1var}, note that $\langle X,I\rangle + \langle X,A\rangle = 1$ and $0 < \langle X ,I \rangle \leq 1/2$, and thus by making
	\[\alpha = \frac{1}{\vartheta(\ov{G})} ,\quad Y = \frac{X \circ I}{\langle X,I\rangle} , \quad \text{and}\quad Z = \frac{X \circ A}{\langle X, A \rangle},\]
	one obtains a feasible solution for the program of $\widetilde{\alpha_0}$ whose objective value is precisely $1/\vartheta(\ov{G})$.
\end{proof}

Note that the second statement in Theorem \ref{thm:easy} is an immediate consequence of Theorem \ref{thm:main}.

\vspace{1.5cc}
\begin{section}
{The parameter \(\alpha_0\) and cuts}
\end{section}

We now apply some of this technology to relate $\alpha_0$ and cuts of the graph~\(G\).
We first prove a resulting relating the minimum eigenvalue of \(A_\alpha\)
to the optimal of a semidefinite relaxation for max-\(k\)-cut and use this
result to derive bounds on \(\alpha_0\).

Consider the SDP relaxation for max-\(k\)-cut by Frieze and Jerrum \cite{FriezeJmaxkcut}:

\sdp{\qquad\qquad \max }{\frac{k - 1}{2k} \langle L, Y \rangle \label{eq:maxkcut}}{Y \succcurlyeq 0 \nonumber\\ & Y_{ii} = 1 \qquad\qquad \, \text{for every \(i \in V\)}, \nonumber\\ & Y_{ij} \geq \frac{-1}{k - 1} \qquad \text{for every \(i,j \in V\)}. \nonumber}

\begin{proposition}
\label{pro:maxkcut}
    Let \(G\) be a graph with \(n\) vertices and \(m\) edges. Let \(\alpha \in [0, 1]\). Let \(\beta^*\) be the optimal value of (\ref{eq:maxkcut}). Then
    \[
    \lambda_{\min}(A_\alpha) \leq \frac{2m}{n} - \frac{(1 - \alpha)}{n} \cdot\frac{2k}{k - 1} \beta^*.
    \]
\end{proposition}
\begin{proof}
    Let \(Y^*\) be an optimal solution for (\ref{eq:maxkcut}).
    As noted in (\ref{eigendual}), we can write \(\lambda_{\min}(A_\alpha)\) as an SDP.
    Set \(X^* \coloneqq Y^* / n\).
    We have that \(X^*\) is positive semidefinite and \(\langle X^*, I \rangle = 1\). Thus, \(X^*\) is feasible for (\ref{eigendual}).
    Using the fact that \(A_\alpha = D - (1 - \alpha) L\), one can compute
    \begin{align*}
        \lambda_{\min}(A_\alpha) 
        &\leq
        \langle A_\alpha, X^* \rangle \\
        &=
        \langle D, X^* \rangle - (1 - \alpha)\langle L, X^* \rangle \\
        &=
        \frac{1}{n}\langle D, Y^* \rangle - \frac{(1 - \alpha)}{n}
        \langle L, Y^* \rangle \\
        &=
        \frac{2m}{n} - \frac{(1 - \alpha)}{n} \cdot\frac{2k}{k - 1} \beta^*. &&\qedhere\\
    \end{align*}
\end{proof}

\begin{corollary}
\label{col:maxkcut}
    Let \(G\) be a graph with \(n\) vertices and \(m\) edges, and let \(\mathrm{MC}_k(G)\) be the value of a maximum
    \(k\)-cut of \(G\). Then
    \[
    \lambda_{\min}(A_\alpha) \leq \frac{2m}{n} - \frac{(1 - \alpha)}{n} \cdot\frac{2k}{k - 1} \mathrm{MC}_k(G),
    \]
    and consequently
    \[
    \alpha_0(G) \geq 1 - \frac{(k - 1)}{k} \frac{m}{\mathrm{MC}_k(G)}.
    \]
\end{corollary}
\begin{proof}
    The first part comes directly from the fact that (\ref{eq:maxkcut}) is
    a relaxation for the maximum \(k\)-cut and so \( \beta^* \geq \mathrm{MC}_k(G)\).
    The second equality holds because when \(\alpha = \alpha_0\), we have
    \(A_\alpha\) is positive semidefinite and so \(\lambda_{\min}(A_\alpha) \geq 0\).
\end{proof}

When \(k = 2\), the SDP in (\ref{eq:maxkcut}) is equivalent to the celebrated max-cut relaxation by Goemans and Williamson~\cite{GoemansW95a}.
This immediately gives an alternate proof of \cite[Proposition~29]{Nikiforov17a}.

\begin{corollary}
    Let \(G\) be a graph with \(n\) vertices and \(m\) edges, and let 
    \(\mathrm{MC}_2(G)\) be the value of a maximum cut of \(G\). Then
    \[
    \lambda_{\min}(A_\alpha) \leq \frac{2m}{n} - \frac{4(1 - \alpha)}{n} \cdot \mathrm{MC}_2(G).
    \]
\end{corollary}

From Corollary \ref{col:maxkcut}, we also obtain

\begin{corollary}
    Let \(G\) be a graph with \(n\) vertices and \(m\) edges, and let 
    \(\mathrm{MC}_2(G)\) be the value of a maximum cut of \(G\). Then
    \[
     \alpha_0(G) \geq 1 - \frac{m}{2\mathrm{MC}_2(G)}.
    \]
\end{corollary}

Finally, we can also derive a alternative proof for \cite[Corollary 11]{NikiforovR17a}.
Suppose \(k = \chi(G)\).
Then \(\mathrm{MC}_k(G) = m\), and,
by Corollary (\ref{col:maxkcut}), we obtain
\[
\alpha_0(G) \geq \frac{1}{\chi(G)}.
\]

\vspace{1.5cc}
\begin{section}
{Copositive formulation}
\end{section}

The well known sandwich theorem proved by Lovász in \cite{LovaszShanon} states that
\[
	\omega(G) \leq \vartheta(\ov{G}) \leq \chi(G).
\]
In light of our result earlier, it is therefore natural to wonder whether ${\alpha_0 \geq 1/\omega}$.
However, this is not true. By Proposition 6 of \cite{NikiforovR17a}, we have that for any triangle-free non-bipartite graph, \(C_5\) for example, it holds that \(\alpha_0 < 1/\omega\).
Inspired by the formulation of $\omega(G)$ in terms of copositive programming (see~\cite{KlerkPasechnikIndepCopProgram}), we introduce below a variant of $\alpha_0$ for which we can prove the inequality of~$\omega(G)$.

Recall that \(\Sds_n\) is the set of symmetric matrices \(n \times n\).
Define the set of copositive matrices as \[\mathcal{C}_n \coloneqq \{X \in \Ss_n \colon \qform{X}{h} \geq 0 \text{ for all } h \geq 0\}.\]
Note that \(\Cc_n\) is a matrix cone that contains the cone of positive semidefinite matrices (see \cite{KlerkPasechnikIndepCopProgram}). Inspired by the formulation \eqref{eq:dual} for \(\alpha_0\), we define
\begin{equation}
    \label{eq:copalpha}
	\alpha_0^{\Cc} \coloneqq \max \{\langle X,-A \rangle : X \in \Cc_n,\ \langle X,L\rangle \leq 
            1 \}.
\end{equation}
Because every positive semidefinite matrix is also copositive we have
 the trivial inequality \(\alpha_0~\leq~\alpha_0^{\Cc}\).

Nikiforov \cite{Nikiforov17a} proved that \(\alpha_0 \leq 1/2\).
We show that this bound also holds for \(\alpha_0^\Cc\), and in addition we show that \(\alpha_0^\Cc\) is lower bounded by $1/\omega$.

\begin{theorem}
    Let \(G = (V, E)\) be a graph with $n$ vertices and $m$ edges, $m>0$. Let $\omega$ be the clique number of $G$. Then
    \[\frac{1}{\omega(G)} \leq \alpha_0^\Cc \leq \frac{1}{2}\]
\end{theorem}
\begin{proof}
	A well-known result due to Motzkin-Strauss \cite{MotzkinS65a} states that
	\[
		\frac{1}{\omega} = \min \{\qform{(\ov{A} + I)}{x} \colon x \in \Rds^n, x \geq 0, \1^\T x = 1. \}.
	\]
	An immediate consequence (also observed in \cite[Corollary 2.4]{KlerkPasechnikIndepCopProgram}) is that if
	\[
		Y \coloneqq \ov{A} + I - \frac{1}{\omega} J,
	\] 
	then $Y$ is copositive. Upon defining
	\[
		X \coloneqq \frac{1}{2m} Y,
	\]
	it is straightforward to verify that $\langle L , X \rangle = 1$, and therefore $X$ is feasible for \eqref{eq:copalpha}, and its objective value is
	\[
		\langle -A , X \rangle = \frac{1}{2m}\cdot \frac{2m}{\omega} = \frac{1}{\omega}.
	\] 
	
    For the other inequality, define the signless Laplacian \(Q\) of \(G\) as
    \[Q \coloneqq \sum_{ij \in E} (e_i + e_j)^{} (e_i + e_j)^\transp\]
    and note that we can write \(\tfrac{1}{2}Q = \tfrac{1}{2} L + A\).
    Now suppose that \(X\) is feasible for the optimization problem from
    \eqref{eq:copalpha}. From the definition of the set of copositive matrices,
    we have that \(\langle \tfrac{1}{2} Q, X \rangle \geq 0\). Hence \(\langle \tfrac{1}{2} L, X \rangle \geq
    \langle X, -A \rangle\).
    Therefore, if \(X\) is an optimal solution for \eqref{eq:copalpha} we obtain
    \[\langle X, -A \rangle \leq \langle \tfrac{1}{2} L, X \rangle \leq \frac{1}{2}.\]
\end{proof}

We do not believe $1/\omega$ is equal to the optimal value of \eqref{eq:copalpha}. This optimum value might be an interesting new graph parameter, and we leave the open problem of computing its value in terms of other graph parameters, or obtaining better bounds to it.

\vspace{1.5cc}
\begin{section}
{Acknowledgements}
\end{section}

We acknowledge conversations about this problem with Marcel K. de Carli Silva, Levent Tunçel, Vinicius Santos, Carla Oliveira, and Pedro Cipriano.

Part of this work was conducted during LAWCG 2022 in Curitiba. We acknowledge the financial support from CAPES-PROEX to attend this event. Thiago Oliveira received a CAPES Master's scholarship. Gabriel Coutinho gratefully acknowledges the support of CNPq and FAPEMIG.

\noindent
{\bf Gabriel Coutinho}\\
Dept. of Computer Science,\\
Federal University of Minas Gerais,\\
Belo Horizonte, Brazil\\
E-mail: {\it gabriel@dcc.ufmg.br}

\vspace{0.1cc}

\noindent
{\bf Thiago Oliveira}\\
Dept. of Computer Science,\\
University of São Paulo,\\
São Paulo, Brazil\\
E-mail: {\it thilio@ime.usp.br}

\end{document}